\documentclass{amsart}[12pt]
\def\doctype{}

\usepackage{latexsym,amssymb}
\usepackage{fancyhdr}
\usepackage{enumitem}
\usepackage{tikz}
\usepackage{hyperref}

\newcommand{\bluebar}[3]{\filldraw[blue] (#1-0.2,#2)--(#1+0.2,#2)--(#1+0.2,#2+#3)--(#1-0.2,#2+#3);}
\newcommand{\pinkbar}[3]{\filldraw[pink] (#1-0.2,#2)--(#1+0.2,#2)--(#1+0.2,#2+#3)--(#1-0.2,#2+#3);}

\newcommand\lam{\lambda}

\newcommand\Z{\mathbb{Z}}
\newcommand{\cF}{\mathcal{F}}
\newcommand{\ap}{\mathrm{AP}}

\newcommand{\comment}[1]{}

\numberwithin{equation}{section}

\fancypagestyle{titlepage}{
\fancyhead[R]{\doctype}
\fancyhead[CL]{}
\cfoot{\vspace{5pt} \thepage}
}

\let\oldsection\section
\newcommand\boldsection[1]{\oldsection{\bf #1}}
\newcommand\starsection[1]{\oldsection*{\bf #1}}
\makeatletter
\renewcommand\section{\@ifstar\starsection\boldsection}
\makeatother

\newtheoremstyle{theorem}
  {12pt}		  % space above
  {0pt}  % space below
  {\sl}  % bofy font
  {\parindent}     % ident - empty=no indent,  \parindent= paragraph indent
  {\bf}  % thm head font
  {. }    % punctuation after thm head
  { }    % space after thm head: `` ``=normal \newline=linebreak
  {}     % thm head specification
\theoremstyle{theorem}
\newtheorem{thm}{Theorem}[section]  % 1st argument is your name for it
\newtheorem{lemma}[thm]{Lemma}     % 2nd argument is what is printed
\newtheorem{cor}[thm]{Corollary}
\newtheorem{conj}[thm]{Conjecture}

\newtheorem{prop}[thm]{Proposition}

\newtheoremstyle{definition}
  {12pt}		  % space above
  {0pt}  % space below
  {}  % bofy font
  {\parindent}     % ident - empty=no indent,  \parindent= paragraph indent
  {\bf}  % thm head font
  {. }    % punctuation after thm head
  { }    % space after thm head: `` ``=normal \newline=linebreak
  {}     % thm head specification
\theoremstyle{definition}

\newtheorem{ex}[thm]{Example}

\newcommand\rk{{\sc Remark.} }

\renewcommand{\proofname}{Proof}

\makeatletter
\renewenvironment{proof}[1][\proofname]{\par
  \pushQED{\qed}%
  \normalfont \partopsep=\z@skip \topsep=\z@skip
  \trivlist
  \item[\hskip\labelsep
        \scshape
    #1\@addpunct{.}]\ignorespaces
}{%
  \popQED\endtrivlist\@endpefalse
}
\makeatother

\makeatletter
\renewcommand*\@maketitle{%
  \normalfont\normalsize
  \@adminfootnotes
  \@mkboth{\@nx\shortauthors}{\@nx\shorttitle}%
  \global\topskip42\p@\relax % 5.5pc   "   "   "     "     "
  \@settitle
  \ifx\@empty\authors \else {\vskip 1em
\vtop{\centering\shortauthors\@@par}} \fi
  \ifx\@empty\@date \else {\vskip 1em \vtop{\centering\@date\@@par}}\fi % MY CHANGE
  \ifx\@empty\@dedicatory
  \else
    \baselineskip18\p@
    \vtop{\centering{\footnotesize\itshape\@dedicatory\@@par}%
      \global\dimen@i\prevdepth}\prevdepth\dimen@i
  \fi
  \@setabstract
  \normalsize
  \if@titlepage
    \newpage
  \else
    \dimen@34\p@ \advance\dimen@-\baselineskip
    \vskip\dimen@\relax
  \fi
} % end \@maketitle
\renewcommand*\@adminfootnotes{%
  \let\@makefnmark\relax  \let\@thefnmark\relax
  \ifx\@empty\@subjclass\else \@footnotetext{\@setsubjclass}\fi
  \ifx\@empty\@keywords\else \@footnotetext{\@setkeywords}\fi
  \ifx\@empty\thankses\else \@footnotetext{%
    \def\par{\let\par\@par}\@setthanks}%
  \fi
\thispagestyle{titlepage}
}
\makeatother

\begin{document}

\title[Families of modular APs]{Families of modular arithmetic progressions with an interval of distance multiplicities}

\author{Peter J.~Dukes}
\address{\rm 
Mathematics and Statistics,
University of Victoria, Victoria, Canada
}
\email{dukes@uvic.ca}

\author{Tao Gaede}
\email{taogaede@uvic.ca}

\date{\today}

\begin{abstract}
Given a family $\mathcal{F}=\{A_1,\dots,A_s\}$ of subsets of $\Z_n$, define $\Delta \cF$ to be the multiset of all (cyclic) distances dist$(x,y)$, where $\{x,y\} \subset A_i$, $x \neq y$, for some $i=1,\dots,s$.  Taking inspiration from a Euclidean distance problem of Erd\H{o}s, we say that $\mathcal{F}$ is \emph{Erd\H{o}s-deep} if the multiplicities of distances that occur in $\Delta \cF$ are precisely $1,2,\dots,k-1$ for some integer $k$.  In the case $s=1$, it is known that a modular arithmetic progression in $\Z_n$ achieves this property (under mild conditions); conversely, APs are the only such sets, except for one sporadic case when $n=6$.  Here, we consider in detail the case $s=2$.  In particular, we classify Erd\H{o}s-deep pairs $\{A_1,A_2\}$ when each  $A_i$ is an arithmetic progression in $\Z_n$.  We also give a construction of a much wider class of Erd\H{o}s-deep families $\{A_1,\dots,A_s\}$ when $s$ is a square integer.
\end{abstract}

\subjclass[2020]{05B07 (primary); 05B05, 05B10 (secondary)}
\maketitle
\hrule

\section{Introduction}
\label{intro}

There is a rich history of problems based on realizing a distance multiset with given properties.
For example, in 1946 Erd\H{o}s asked \cite{ErdosDD} how many
distinct distances are determined by $k$ points in the plane.  A conjectured lower bound of $\Omega(k/\sqrt{\log k})$ distinct distances, attained by vertices of a square grid, has attracted considerable attention, \cite{GIS}.  A celebrated bound of $\Omega(k/\log k)$ was found by Guth and Katz in \cite{GK}.

A separate question of Erd\H{o}s \cite{ErdosDeep1,ErdosDeep2} concerns the placement of $k$ points in the plane, no three on a line and no four on a circle, such that the multiplicities of the $\binom{k}{2}$ pairwise distances are $1,2,\dots,k-1$.  For instance, a parallelogram with side lengths $1$ and $\sqrt{2}$ at angle $45^\circ$ has diagonal lengths $1,\sqrt{5}$.  So, distance $d_i$ occurs exactly $i$ times, where in this case $(d_1,d_2,d_3)=(\sqrt{5},\sqrt{2},1)$.  Solutions have been found for other small values of $k$, most notably an $8$ point configuration of Palasti \cite{Palasti}.  However, it is conjectured that no placement of points with this property is possible for $k \ge k_0$. Perhaps $k_0=9$; in \cite{BGMMPS} a large number of plausible nine-point configurations were ruled out by computer search.

If we were to allow collinear points, then an arithmetic progression of $k$ points on a line has the desired distance multiplicities.  Likewise, if we drop the restriction that no four points lie on a circle, it is possible to take an (angular) arithmetic progression of $k$ points on a semicircle.  For this reason, a set of $k$ points with distance multiplicities $1,2,\dots,k-1$ is also called a \emph{crescent configuration}, \cite{BGMMPS}.

Problems concerning distances and their multiplicities can be posed in different metric spaces, including discrete settings.  For instance, the classical problem on distinct distances has been studied in finite vector spaces, \cite{IR}.

Our context of interest is $\Z_n$, the ring of integers modulo $n$, with the usual norm $|x|_n:=
\min(x,-x)$, where each of $\pm x$ is reduced in $\{0,1,\dots,n-1\}$. We measure distance in $\Z_n$ as $\mathrm{dist}(x,y)=|x-y|_n$.  For a $k$-element set $A \subseteq \Z_n$, we let $\Delta A$ denote the multiset of $\binom{k}{2}$ pairwise distances that occur in $A$.  That is,
$\Delta A = \{ \mathrm{dist}(x,y) : \{x,y\} \subseteq A, x \neq y\}$.
We are interested here in the {\bf multiplicities} with which elements occur in $\Delta A$.  There is some precedent for studying distance multiplicities.  An $(n,k,\lambda)$-\emph{difference set} is a $k$-element set $K \subseteq \Z_n$, say $K =\{x_1,\dots,x_k\}$, such that every nonzero element of $\Z_n$ occurs exactly $\lambda$ times as a difference $x_i-x_j$, $1 \le i,j \le k$. When $n$ is odd, this is equivalent to $\Delta K$ attaining every value from $1$ to $(n-1)/2$ with the same multiplicity $\lam$.  For example, $\{1,3,4,5,9\} \subseteq \Z_{11}$ is an $(11,5,2)$-difference set, which we note in passing is related to many combinatorial objects, including Steiner systems, Mathieu groups, and Golay codes; see \cite{biplane}.

In a different direction, 
a set $E \subseteq \Z_n$ is \emph{deep} if the multiplicities of distances in $\Delta E$ are all distinct; see Toussaint's book \cite{Toussaint}, which offers an excellent mathematical survey of musical rhythms.  It is argued there that deep sets are musically interesting.  Adding extra structure, and referencing the consecutive multiplicity problem from \cite{ErdosDeep1,ErdosDeep2} above, the authors of \cite{DistGeom} define a $k$-set $E \subseteq \Z_n$ to be \emph{Erd\H{o}s-deep} if the distances that occur in $\Delta E$ have multiplicities $1,2,\dots,k-1$.  That is,  $E$ is Erd\H{o}s-deep if and only if, for every $i \in \{1,2,\dots,k-1\}$, there is exactly one positive number $d_i$ such that $$|\{\{x,y\} \subseteq E : \mathrm{dist}(x,y)=d_i\}| = i.$$
Along with some musical motivation, the authors give a classification of Erd\H{o}s-deep sets.  Arithmetic progressions (mod $n$) are nearly the whole story.

\begin{thm}[see \cite{DistGeom}]
Let $E \subseteq \Z_n$ be an Erd\H{o}s-deep set.  Then $E$ is either an arithmetic progression with generator $g$ satisfying  $|E| \le \lfloor n/2\gcd(n,g) \rfloor +1$ or $n=6$ and $E$ is a translate of $\{0,1,2,4\}$.
\end{thm}

One natural way to extend the preceding distance problems is to allow several sets, and consider distances within the same set.  For example, an $(n,k,\lambda)$-\emph{difference family} is a collection $\mathcal{F}$, such that each $K \in \mathcal{F}$ is a $k$-subset of $\Z_n$, and the `internal' differences $x-y$, $x,y \in K \in \mathcal{F}$ attain every nonzero value in $\Z_n$ with multiplicity $\lambda$.  In general, for a family $\cF =\{A_1,\dots,A_s\}$ of subsets of $\Z_n$ (of possibly different sizes), we define $\Delta \cF$ to be the multiset union of all $\Delta A_j$, $j=1,\dots,s$.

To generalize the Erd\H{o}s-deep problem for single sets, let us say that $\mathcal{F}$ is an \emph{Erd\H{o}s-deep family} if $\Delta \cF$ achieves precisely the  multiplicities $1,2,\dots,k-1$ for some $k$.

\begin{ex}
\label{ED-infinite}
For any $n \ge 8$, $\{\{0,1,2\},\{0,2,4\}\}$ is an Erd\H{o}s-deep family of subsets of $\Z_n$.  The internal distances within sets are $1,1,2,2,2,4$, realizing each multiplicity $1,2,3$ exactly once.  This family has the additional property that each constituent set is an arithmetic progression, and hence an Erd\H{o}s-deep set.  In fact, the family is also Erd\H{o}s-deep when interpreted in $\Z_7$, the only minor adjustment being that distance 4 becomes 3.  However, the condition fails over $\Z_6$, since in that case distance $2$ occurs four times.  %Note that multiplying every element by a unit in $\Z_n$ yields the same multiplicities.
\end{ex}

Our main result is a classification of Erd\H{o}s-deep families of two modular arithmetic progressions.
Define $\ap_n(g,k):=\{0,g,2g,\dots,(k-1)g\} \subset \Z_n$, which we informally call a modular $k$-term AP.

\begin{thm}
\label{main}
Suppose $\{\ap_n(g_1,k_1),\ap_n(g_2,k_2)\}$ is an Erd\H{o}s-deep family in $\Z_n$ with  $k_1 \ge k_2 \ge 3$ and $\gcd(n,g_1,g_2)=1$.  Then either $k_1=k_2=3$, or $(n,k_1,k_2) \in \{ (13,6,4), (19,7,6), (31,11,9)\}$.
\end{thm}

The next section sets up some notation and gives an outline of the proof of Theorem~\ref{main}, including a look at the three sporadic cases.  Then, Section~\ref{proof} covers various estimates and other details needed to complete the proof.  Very broadly, all but a finite set of parameters can be ruled out by elementary combinatorial estimates, and a computer search handles what remains.  Section~\ref{construction} gives a construction to show that Erd\H{o}s-deep families of APs with large modulus $n$ are abundant, in contrast with the main result, when the size of the family is a square integer.  We  conclude in Section~\ref{sec:discussion} with a brief discussion of some next steps for the research topic.

\section{Set-up and outline}

\subsection{Three sporadic cases}

Recall that Example~\ref{ED-infinite} gives an Erd\H{o}s-deep family $\{\{0,1,2\}$, $\{0,2,4\}\}$ 
$\subseteq \Z_n$ for each $n \ge 7$.  It is natural to refer to this as the `geometric' infinite family.  We next examine the three sporadic cases in which other Erd\H{o}s-deep pairs of APs can exist.

\begin{ex}
\label{edfam-13}
Let $n=13$ and consider $\cF=\{\ap_n(1,6),\ap_n(3,4)\} = \{\{0,1,2,3,4,5\},\{0,3,6,9\}\}$.  Let $A_1$ and $A_2$ be the two APs in the order given.  Using exponential notation to denote multiplicities, $\Delta A_1 = \{1^5, 2^4, 3^3, 4^2, 5^1\}$ and $\Delta A_2 = \{3^3, 6^2, 4^1\}$.
The multiset union of these is  
$$\Delta \cF = \{1^5,2^4,3^6,4^3, 5^1,6^2\},$$
and we see that the multiplicities are precisely the values $1,2, \dots,6$.
\end{ex}

\begin{ex}
\label{edfam-19}
Let $n=19$ and consider $\cF=\{\ap_n(1,7),\ap_n(4,6)\}$.  We have 
$$\Delta \cF = \{1^7, 2^5, 3^6, 4^8, 5^2, 6^1, 7^3, 8^4\},$$
with multiplicities  $1,2, \dots,8$.  
\end{ex}

\begin{ex}
\label{edfam-31}
Let $n=31$ and consider $\cF=\{\ap_n(1,11),\ap_n(13,9)\}$.  The distance multiplicities are $1,2, \dots, 13$.
\end{ex}

Figure~\ref{histograms} displays histograms of multiplicities for the three preceding examples.  Multiplicities of distances in the APs are shown in blue and red, respectively.
On visual inspection, the multiplicities combine to form an interval starting at 1.
\begin{figure}[htbp]
\begin{center}
\begin{tikzpicture}[scale=0.5]
\bluebar{1}{0}{5}
\bluebar{2}{0}{4}
\bluebar{3}{0}{3}
\bluebar{4}{0}{2}
\bluebar{5}{0}{1}
\pinkbar{3}{3}{3}
\pinkbar{6}{0}{2}
\pinkbar{4}{2}{1}
\draw (-1,0)--(7.2,0);
\draw (0,-1)--(0,6.5);
\foreach \a in {1,2,...,6}
\node at (\a,-0.5) {$\a$};
\end{tikzpicture}
\hspace{2cm}
\begin{tikzpicture}[xscale=0.5,yscale=0.4]
\bluebar{1}{0}{6}
\bluebar{2}{0}{5}
\bluebar{3}{0}{4}
\bluebar{4}{0}{3}
\bluebar{5}{0}{2}
\bluebar{6}{0}{1}
\pinkbar{4}{3}{5}
\pinkbar{8}{0}{4}
\pinkbar{7}{0}{3}
\pinkbar{3}{4}{2}
\pinkbar{1}{6}{1}
\draw (-1,0)--(9.2,0);
\draw (0,-1)--(0,8);
\foreach \a in {1,2,...,8}
\node at (\a,-0.6) {$\a$};
\end{tikzpicture}

\begin{tikzpicture}[xscale=0.5,yscale=0.3]
\bluebar{1}{0}{10}
\bluebar{2}{0}{9}
\bluebar{3}{0}{8}
\bluebar{4}{0}{7}
\bluebar{5}{0}{6}
\bluebar{6}{0}{5}
\bluebar{7}{0}{4}
\bluebar{8}{0}{3}
\bluebar{9}{0}{2}
\bluebar{10}{0}{1}
\pinkbar{13}{0}{8}
\pinkbar{5}{6}{7}
\pinkbar{8}{3}{6}
\pinkbar{10}{1}{5}
\pinkbar{3}{8}{4}
\pinkbar{15}{0}{3}
\pinkbar{2}{9}{2}
\pinkbar{11}{0}{1}
\draw (-1,0)--(15.5,0);
\draw (0,-1)--(0,13);
\foreach \a in {1,2,...,15}
\node at (\a,-0.6) {\small $\a$};
\end{tikzpicture}
\end{center}
\caption{Multiplicity histograms for the families in Examples~\ref{edfam-13}, \ref{edfam-19} and \ref{edfam-31}.}
\label{histograms}
\end{figure}
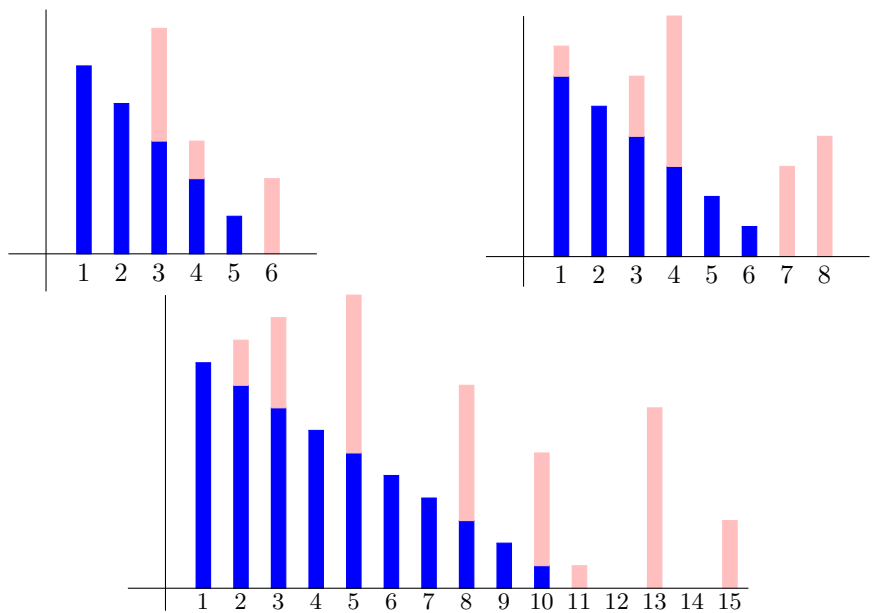

Our main result, Theorem~\ref{main}, argues that these three examples and the geometric infinite family in Example~\ref{ED-infinite} are the only Erd\H{o}s-deep pairs of APs, up to translation and scaling by units.

\subsection{A geometric interpretation}

Consider two light beams inside a circle, originating from a common source $O$ on the boundary. The beams begin with intensities $k_1,k_2$. Thereafter, each time a beam touches the boundary, its intensity decreases by one and it reflects in the usual way, continuing until its intensity reaches zero. The Erd\H{o}s-deep condition corresponds to a configuration of beams in which each value $1,2,\dots,k-1$ occurs as a total intensity at some distance from $O$. Figure~\ref{beams} illustrates this for the family from Example~\ref{edfam-19}. For visual effect, beams have been spread out by scaling generators by a common factor; this of course has no effect on the set of multiplicities.

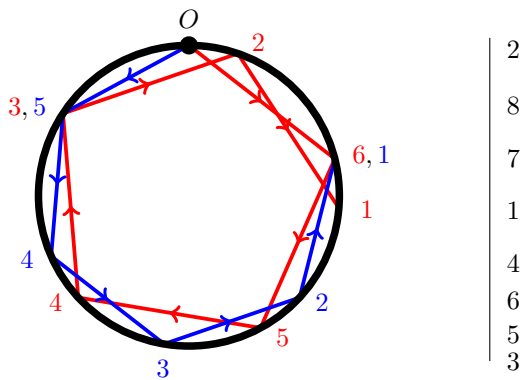
\begin{figure}[htbp]
\begin{center}
\begin{tikzpicture}[scale=0.5]
\usetikzlibrary{decorations.markings}
\begin{scope}[very thick,decoration={
markings,
mark=at position 0.5 with {\arrow{>}}}] 
\foreach \a in {0,1,2,...,5}
\draw[red,line width=0.5mm,postaction={decorate}] (90-4*\a*18.947:4)--(90-4*18.947-4*\a*18.947:4);
\foreach \a in {0,1,2,...,4}
\draw[blue,line width=0.5mm,postaction={decorate}] (90+3*\a*18.947:4)--(90+3*18.947+3*\a*18.947:4);
\draw[line width=1mm] (0,0) circle [radius=4];
\filldraw (90:4) circle [radius=.2];
\node at (0,4.7) {$O$};
\node[anchor=west] at (90-18.947:4.3) {$\color[rgb]{1,0,0}2\color[rgb]{0,0,0}$};
\node[anchor=west] at (90-4*18.947:4.2) {$\color[rgb]{1,0,0}6\color[rgb]{0,0,0},\color[rgb]{0,0,1}1\color[rgb]{0,0,0}$};
\node[anchor=west] at (90-5*18.947:4.3) {$\color[rgb]{1,0,0}1\color[rgb]{0,0,0}$};
\node[anchor=east] at (90+3*18.947:4.2) {$\color[rgb]{1,0,0}3\color[rgb]{0,0,0},\color[rgb]{0,0,1}5\color[rgb]{0,0,0}$};
\node[anchor=east] at (90+6*18.947:4.2) {$\color[rgb]{0,0,1}4\color[rgb]{0,0,0}$};
\node[anchor=west] at (90-8*18.947:4.3) {$\color[rgb]{1,0,0}5\color[rgb]{0,0,0}$};
\node[anchor=north] at (90+9*18.947:4.15) {$\color[rgb]{0,0,1}3\color[rgb]{0,0,0}$};
\node[anchor=west] at (90-7*18.947:4.2) {$\color[rgb]{0,0,1}2\color[rgb]{0,0,0}$};
\node[anchor=east] at (90+7*18.947:4.2) {$\color[rgb]{1,0,0}4\color[rgb]{0,0,0}$};
\end{scope}
\draw (8,-4.4)--(8,4.2);
\node[anchor=west] at (8.2,3.9) {$2$};
\node[anchor=west] at (8.2,2.4) {$8$};
\node[anchor=west] at (8.2,1) {$7$};
\node[anchor=west] at (8.2,-0.4) {$1$};
\node[anchor=west] at (8.2,-1.8) {$4$};
\node[anchor=west] at (8.2,-2.8) {$6$};
\node[anchor=west] at (8.2,-3.7) {$5$};
\node[anchor=west] at (8.2,-4.4) {$3$};
\end{tikzpicture}
\end{center}
\caption{Light beam model for an Erd\H{o}s-deep pair with $(k_1,k_2)=(7,6)$.}
\label{beams}
\end{figure}
This interpretation and the histograms above reveal the importance of estimating the size of the intersection of two modular arithmetic progressions.  This is considered later, in Section~\ref{sec:ap-int}.

\subsection{Elementary arithmetic constraints on the parameters}

Suppose we have an Erd\H{o}s-deep pair of arithmetic progressions of lengths $k_1 \ge k_2 \ge 3$, and suppose the multiplicity list for the pair is $1,2,\dots,k-1$.  Summing multiplicities, we have the identity
\begin{equation}
\label{basic-equation}
k(k-1)=k_1(k_1-1)+k_2(k_2-1).
\end{equation}
It is useful in what follows to consider the value $t:=k-k_1$, which we note is the number of distances in the family that are not in the longer AP.  Using this in \eqref{basic-equation} produces the parameterization
\begin{align}
\label{useful-param}
k_1 &= \frac{k_2(k_2-1)}{2t}-\frac{t-1}{2}.
%\nonumber
%k &= \frac{k_2(k_2-1)}{2t}+\frac{t+1}{2}.
\end{align}
That is, $k_1$ and $k$ are uniquely determined from $k_2$ and $t$.  Next, we present an elementary bound that is useful for our later analysis.

\begin{lemma}
\label{three-sevenths}
We have $t \le \frac{3}{7}(k_2-1)$ unless $k_1=k_2=3$
\end{lemma}

\begin{proof}
Suppose $t>\frac{3}{7}(k_2-1)$.  Then, working from \eqref{useful-param},
$k_2(k_2-1) = 2k_1 t + t(t-1)$ implies
\begin{equation}
\label{three-sevenths-intermediate}
k_2 > \tfrac{6}{7}k_1 + \tfrac{3}{7}(t-1).
\end{equation}
When $k_1 \ge k_2+1$, this gives
$$k_2 > \tfrac{6}{7}k_2 + \tfrac{9}{49}(k_2-1) + \tfrac{3}{7} = \frac{51k_2+12}{49},$$
a contradiction.  When $k_1=k_2$, \eqref{three-sevenths-intermediate} becomes
$$k_2 > \tfrac{6}{7}k_2 + \tfrac{9}{49}(k_2-1) - \tfrac{3}{7} = \frac{51k_2-30}{49}.$$
This forces $k_2<15$.  It is straightforward to check that there are no integer solutions to $2k_2(k_2-1)=k(k-1)$
with $3<k_2<15$.  It follows that $k_1=k_2 \le 3$.
\end{proof}

We conclude this section with an important discussion on common divisors. 
Suppose $\{\ap_n(g_1,k_1),$ $\ap_n(g_2,k_2)\}$ is an Erd\H{o}s-deep pair of APs with $\gcd(n,g_1,g_2)=g$.
Every distance is a multiple of $g$, so
$\{\ap_{n/g}(g_1/g,k_1), \ap_{n/g}(g_2/g,k_2)\}$ has the same distance multiplicities, and is therefore also an Erd\H{o}s-deep pair.  For this reason, we henceforth assume that $\gcd(n,g_1,g_2)=1$.

Under this assumption, we also argue that $\gcd(n,g_1)$ may be assumed to equal 1.  Suppose a prime $p$ divides both $n$ and $g_1$ but not $g_2$.  Let $A_i = \ap_n(g_i,k_i)$ for $i=1,2$.  Let $D$ be the set of distances occurring in both $\Delta A_1$ and $\Delta A_2$.  Since $\Delta A_1$ consists of multiples of $p$, we have $|D| \le \lfloor (k_2-1)/p \rfloor \le (k_2-1)/2$.  On the other hand, $\cF=\{A_1,A_2\}$ being Erd\H{o}s-deep requires 
\begin{align*}
|D|&=k_1+k_2-k-1\\ 
&= k_2-1-t \ge \tfrac{4}{7}(k_2-1),
\end{align*}
where we have applied inclusion-exclusion and Lemma~\ref{three-sevenths}.  From this contradiction, $\gcd(n,g_1)=1$.
Now, upon multiplication by $g_1^{-1} \pmod{n}$, our problem is reduced to determining parameters $k_1, k_2, g_2,n$ with $k_1 \ge k_2$ 
%and $\gcd(g_2,n)=1$ 
such that $\{\ap_n(1,k_1)$, $\ap_n(g_2,k_2)\}$ is Erd\H{o}s-deep; that is, the internal distance multiplicities are $\{1,2,\dots,k-1\}$, where $k$ is uniquely determined from \eqref{basic-equation}.

\section{Classification for pairs}
\label{proof}

As seen in the previous section, the pair $\{\ap_n(1,k_1),\ap_n(g_2,k_2)\}$ can be Erd\H{o}s-deep only when many distances in the second AP are less than $k_1$.  This naturally leads us to consider the problem of bounding the number of times a modular AP hits a given interval.

\subsection{The intersection of a modular AP with an interval}
\label{sec:ap-int}

%Let $I=[a,b] \subset R$ be a closed interval of length $\ell=b-a$.  It is easy to see that $|g\Z \cap I| %\in \{\lfloor \frac{\ell}{g} \rfloor, \lceil \frac{\ell}{g} \rceil\}$.  Now, suppose $\ell \le n$ and let $%\overline{I}$ be the image of $I \cap \Z$ under the projection $\Z \rightarrow \Z_n$.

Let $a,b \in \Z$ with $a<b$.  Put $\ell:=b-a+1$.  It is easy to see that the number of multiples of $g$ in the interval $I=[a,b]$ can assume only two values, namely
\begin{equation}
\label{ap-interval}
|g\Z \cap I|  \in \{\lfloor \tfrac{\ell}{g} \rfloor, \lceil \tfrac{\ell}{g} \rceil\}.
\end{equation}
The same holds for any translate of $g\Z$.

If $\ell \le n$, the interval $I=[a,b]$ can be unambiguously interpreted as an $\ell$-subset of $\Z_n$. With this in mind, we 
%slightly abuse terminology and 
refer to $I$ as an interval of size $\ell$ in $\Z_n$.  For instance, the `open ball' $\{x \in \Z_n: |x|_n <r\}$ is an interval $\{-(r-1),\dots,-1,0,1,\dots,r-1\}$ of size $2r-1$ in $\Z_n$.
 
Given an interval $I \subseteq \Z_n$ and a modular AP, say $A=\ap_n(g,k)$, we define the $I$-\emph{hitting sequence} of $A$ as $\mathbf{w}=(w_0,w_1,w_2,\dots,w_{k-1})$, where
$$w_i = 
\begin{cases} 
1 & \text{if~} gi \in I, \\
0 & \text{otherwise}.
\end{cases}
$$
We note that $\mathbf{w}$ resembles a `Sturmian word'; see for instance \cite{Arnoux}.
The number of ones in $\mathbf{w}$ is an upper bound on $|\ap_n(g,k) \cap I|$.
When $gk \ge n$, there are roughly $(gk/n)(\ell/g) = k \ell /n$ ones in $\mathbf{w}$; the exact value depends on the location of $I$ relative to the AP.   In a little more detail, the word $\mathbf{w}$ consists of alternating runs of ones and zeros, the lengths of which are governed by \eqref{ap-interval}.  A precise statement follows.

\begin{lemma}
Let $I \subseteq \Z_n$ be an interval of size $\ell$, and let $\mathbf{w}$ be the $I$-hitting sequence of $\ap_n(g,k)$.
We have, for some integer $p \ge 0$,
$\mathbf{w} = 1^{a_0} 0^{b_0} 1^{a_1} 0^{b_1} \cdots 1^{a_p} 0^{b_p},$
where
\begin{itemize}
\item
either $a_0 = 0$ and $0<b_0 \le  \lceil \frac{n-\ell}{g} \rceil$, or $0<a_0 \le \lceil \frac{\ell}{g} \rceil$ and $b_0 \in \{\lfloor \frac{n-\ell}{g} \rfloor, \lceil \frac{n-\ell}{g} \rceil\};$
\item
for $j=1,\dots,p-1$, 
$a_j \in \{\lfloor \tfrac{\ell}{g} \rfloor, \lceil \tfrac{\ell}{g} \rceil\}$ and $b_j \in \{\lfloor \tfrac{n-\ell}{g} \rfloor, \lceil \tfrac{n-\ell}{g} \rceil\};$ and
\item
either $b_p = 0$ and $0<a_p \le  \lceil \frac{\ell}{g} \rceil$, or $0<b_p\le \lceil \tfrac{n-\ell}{g} \rceil$ and $a_p \in \{\lfloor \frac{\ell}{g} \rfloor, \lceil \frac{\ell}{g} \rceil\}.$
\end{itemize}
\end{lemma}

We omit the proof, but remark that it follows easily from partitioning $\ap_n(g,k)$ into `passes' around $\Z_n$, applying \eqref{ap-interval} to $I$ and $\Z_n \setminus I$ in each pass.  Figure~\ref{ap-hits} illustrates how, for varying $g$ (shown horizontally), the supports of $\mathbf{w}$ partition into nearly equal-sized intervals (shown as vertical slices).

\begin{figure}[htbp]
\begin{center}
\includegraphics[width=15cm]{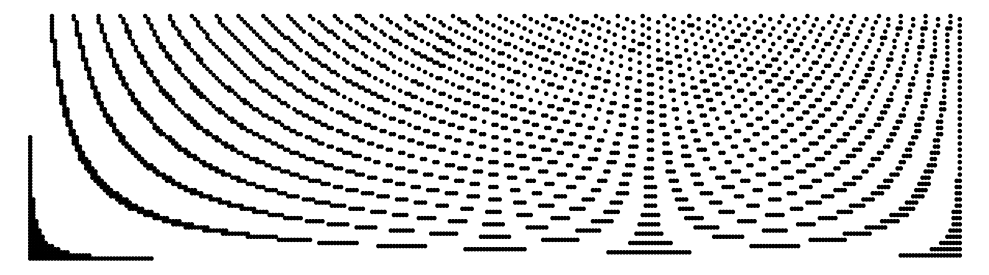}
\end{center}
\caption{Lattice points $(g,x) \in [1, n/2] \times [1,k)$ for which $|gx|_n \le r$; 
shown here for $n=601$, $k=80$, $r=40$.}
\label{ap-hits}
\end{figure}

Suppose $\mathbf{w} \neq 1^k$.  Using the inequality $\frac{a+c}{b+d} \le \max(\frac{a}{b},\frac{c}{d})$ and considering cases, we get
\begin{equation}
\label{ratio-bound}
\frac{|\ap_n(g,k) \cap I|}{|\ap_n(g,k) \setminus I|} =\frac{a_0+a_1+\dots+a_p}{b_0+b_1+\dots+b_p} \le
\begin{cases}
a_0/b_0 & \text{if}~p=0,\\
\lceil \tfrac{\ell}{g} \rceil/b_0 & \text{if}~p=1~\text{and}~a_0=0, \\
(a_0+\lceil \tfrac{\ell}{g} \rceil) / \lfloor \tfrac{n-\ell}{g} \rfloor &\text{otherwise}.
\end{cases}
\end{equation}
This analysis can be strengthened somewhat with more information on the generator $g$, the number of passes $p$, or on the suffix parameters $a_p,b_p$.  For instance, when $g$ is fixed and $p$ is large, the fraction in \eqref{ratio-bound} approaches $\ell/(n-\ell)$. Or, $g \ge \ell$ implies $\mathbf{w}$ has no two consecutive ones, and so in this case the quantity is at most $1$.  
We use \eqref{ratio-bound} and some minor variations in the arguments that follow.

\subsection{Upper bounds on $n$ via distinct multiplicities}
\label{sec:n-upper}

Here, we show that, in an Erd\H{o}s-deep family $\{\ap_n(1,k_1),\ap_n(g_2,k_2)\}$, the modulus $n$ is at most linear in each of $k_1$ and $k_2$.
The general idea is that a large modulus would not allow enough common distances in the two APs.  The hitting sequence structure discussed in Section~\ref{sec:ap-int} is helpful for our estimates.

%Let $\beta:=n/k_1$.  The next result makes use of hitting sequence structure to obtain a bound $\beta<6$.
%The rough idea of the proof is that if $n$ is too large relative to $k_1$, then our two modular %arithmetic progressions won't have enough common distances.

\begin{prop}
\label{beta-upper-bound}
An Erd\H{o}s-deep pair of APs of lengths $k_1 \ge k_2>3$ satisfies $n<6k_1$.
\end{prop}

\begin{proof}
Write $g_2=g$ for simplicity.  
Let $A=\ap_n(g,k_2)=\{0,g,2g,\cdots,(k_2-1)g\}$.  Recall that $A':=A \setminus \{0\}$ is the set of differences that occur in $A$.  Let $B=B_{k_1}=\{x \in \Z_n : |x|_n<k_1\}$ and let $\mathbf{w}$ be the $B$-hitting sequence of $A'$.

Suppose for contradiction that $n \ge 6k_1$.  Consider the ratio $|A' \cap B|/|A' \setminus B|$.  The denominator equals $t=k-k_1$; the numerator equals $k_1+k_2-k-1$, by inclusion-exclusion.  If we can show the ratio is at most $1$, then $2t \ge k_2-1$, which contradicts Lemma~\ref{three-sevenths}.

Suppose first that $\mathbf{w}$ has at least two blocks of ones (separated by at least one zero). If $g>2k_1-1=|B|$, then each block of ones in $\mathbf{w}$ has length one.  This implies $|A' \cap B| \le |A' \setminus B|$.  On the other hand, if $g \le 2k_1-1$, then \eqref{ratio-bound} gives
$$\frac{|A' \cap B|}{|A' \setminus B|} \le \frac{\lfloor  \tfrac{k_1-1}{g} \rfloor +  \lceil \frac{2k_1-1}{g} \rceil}{ \lfloor \frac{4k_1-1}{g} \rfloor} \le \frac{\lfloor  \tfrac{k_1-1/2}{g} \rfloor +  \lceil \frac{2k_1-1}{g} \rceil}{ \lfloor \frac{4k_1-2}{g} \rfloor}  \le 1,$$
where in the last step we have applied the straightforward estimate $\lfloor 4x \rfloor \ge \lceil 2x \rceil +\lfloor x \rfloor$ for $x \ge \frac{1}{2}$, with $x$ taking the role of $\tfrac{k_1-1/2}{g}$. 

Suppose now that $\mathbf{w}=1^{a_0}0^{b_0}$, which is equivalent to $g(k_2-1) \le n-k_1$.  
Recall that $|A \setminus B|=t$, a positive integer.  If $t\ge 2$, then $g(k_2-2),g(k_2-1) \in A \setminus B$.  The multiplicities of these elements in $\Delta A$ are $2$ and $1$, respectively.  But elements $k_1-2,k_1-1$ have multiplicities $2$ and $1$ as distances in $\ap_n(1,k_1)$.  Therefore, we must have $k_1-2,k_1-1 \in \Delta A$, which implies $g=1$, a contradiction.  

Finally, suppose $p=1$ and $t=1$. Since $g$ is the least element of $\Delta A$, the total multiplicity of $g$ is $(k_1-g)+(k_2-1) = k_1$.  This implies $g=k_2-1$.  Now, because $k_2>3$, the total multiplicity of distance 
$2g$ is $(k_1-2g)+(k_2-2) = k_1-k_2$.  But $g+1 \not\in A$, so the total multiplicity of $g+1$ is $k_1-(g+1)=k_1-k_2$.  This repeated multiplicity is another contradiction.
\end{proof}

%\rk  With some mild additional assumptions on $g,k_1,k_2$, the upper bound on $\beta$ can be %made slightly smaller, approaching $5$ for $k_1 \gg g$.

Before stating our next bound on $n$, it is convenient to estimate the number of small gaps between distances in a modular AP.

\begin{lemma}
\label{two-thirds-bound}
Suppose $n>18k$ and $1<g \le n/2$.
Define $D(g,k,n):=\{|gx|_n - |gy|_n : 1 \le x,y < k\}$.  Then
$$|D(g,k,n) \cap \{1,2,\dots,k-1\}| < 2k/3.$$
\end{lemma}

\begin{proof}
There are two possibilities for a positive element in $D$: either
$|gx|_n - |gy|_n = |g(x-y)|_n$ or $|g(x+y)|_n$. 
%At least one of each type of these elements occurs, since otherwise the given intersection size is at most $k/2$.  
So, consider the ball
$B_k=\{x \in \Z_n:|x|_n<k\}$, and the $B_k$-hitting sequence $\mathbf{w}$ for $\ap_n(g,2k-2)$.
If $\mathbf{w}$ has no subword of the form `$11$' or `$101$', the conclusion is automatic.
The subword  `$11$' implies $g<2k-1$, and `$101$' implies $g>n/2-k$.

{\sc Case 1}: $1 < g \le 2k-2$.  Let $A=\ap_n(g,2k-2)$.  If $\mathbf{w}$ contains only one block of ones, then $g\ge 2$ implies $|A \cap B_k| \le k/2$.  Otherwise, from \eqref{ratio-bound},
\begin{equation}
\label{case1}
\frac{|A \cap B_k|}{|A \setminus B_k|} \le \frac{\lceil \tfrac{k-1}{g} \rceil + \lceil \tfrac{2k-1}{g} \rceil}{\lfloor \tfrac{n-2k+1}{g} \rfloor} \le \frac{3\lceil \tfrac{k}{g} \rceil}{\lfloor \tfrac{16k}{g} \rfloor} \le  \frac{3}{8}.
\end{equation}
In the last step, we have used the simple estimate $\lfloor 2x \rfloor \ge \lceil x \rceil$ for $x \ge \frac{1}{2}$.  From \eqref{case1}, we have $|A \cap B_k| \le \frac{1}{3}(2k-3) < 2k/3$.

%We have $|A_0|<\lceil \frac{k-1}{g} \rceil$, and 
%$|A_j \setminus B_k| \ge \lfloor \frac{n-2k+1}{g} \rfloor \ge \lfloor \frac{16k+1}{g} \rfloor \ge 
%\lceil \frac{14k}{g} \rceil \ge 7\lceil \frac{2k-1}{g} \rceil \ge 7|A_j \cap B_k|$.

{\sc Case 2}: $\lfloor n/2 \rfloor-k+1 \le g < \lfloor n/2 \rfloor$.  Put $h=-2g=|2g|_n$, and note that $1<h \le 2k-2$.  It is convenient to partition $A$ into even and odd multiples of $g$, namely $A=A_{\text{even}} \cup A_{\text{odd}}$, where 
$A_{\text{even}} = \ap_n(h,k-1)$ and $A_{\text{odd}} = g+\ap_n(h,k-1)$.  Arguing similarly as in Case 1, we have
$|A_{\text{even}} \cap B_k| < k/3$.  For the other half, we observe that the hitting sequence for $A_{\text{odd}}$ begins with at least $\lfloor (\tfrac{n}{2}-k)/g \rfloor$ zeros.  Using \eqref{ratio-bound} again gives 
$$\frac{|A_{\text{odd}} \cap B_k|}{|A_{\text{odd}} \setminus B_k|} \le \max \left(
\frac{2\lceil \tfrac{k}{h} \rceil}{\lfloor \tfrac{8k}{h} \rfloor},
\frac{3\lceil \tfrac{k}{h} \rceil}{\lfloor \tfrac{16k}{h} \rfloor} \right)  
\le \frac{1}{2}.$$
So $|A_{\text{odd}} \cap B_k| < k/3$.  Combining the even and odd multiples then gives  $|A \cap B_k| \le 2k/3$.

For the final case, we return to directly examining $D(g,k,n)$.

{\sc Case 3}: $g=\lfloor n/2 \rfloor$.  When $n$ is even, the only possible elements in $D(g,k,n)$ are $0$ and $n/2$.  For odd $n$,
we have $2g \equiv -1 \pmod{n}$, and so the positive values in $D(g,k,n)$ are contained in the union of intervals $(0,\tfrac{k-1}{2}) \cup [n/2-k,n/2]$.  From the assumption $n> 18k$, the latter interval is disjoint from $\{1,\dots,k-1\}$, meaning our intersection is at most $\frac{k-1}{2}$, and hence less than $\frac{2k}{3}$.
\end{proof}

\rk
The constant 18 is likely far from best possible.  The conclusion is perhaps still true for $n>10k$.

\begin{prop}
\label{prop:n-upperbound}
An Erd\H{o}s-deep pair of APs of lengths $k_1 \ge k_2 >3$ satisfies $n \le 18k_2+36$.
\end{prop}

\begin{proof}
Suppose for contradiction that $n > 18k_2+36$ and that $\{A_1,A_2\}$ is Erd\H{o}s-deep, where $A_1 = \ap_n(1,k_1)$ and $A_2=\ap_n(g_2,k_2)$.  Consider the set $S$ of distances in both $\Delta A_1$ and $\Delta A_2$.  For a shared distance 
$|xg_2|_n \in S$, its total multiplicity in $\Delta \cF$ equals $(k_1-|xg_2|_n)+(k_2-x)$.  Exactly $t=k-k_1$ elements of $S$ have multiplicity at least $k_1$ in $\Delta \cF$.  Since there are $t$ distinct distances in $\Delta A_2$ but not $\Delta A_1$, there are $k_2-2t-1$ distances in $S$
that have total multiplicity less than $k_1$.  Suppose $|xg_2|_n \in S$ is one such distance.
Let $|yg_2|_n$ be the distance with multiplicity in $\Delta A_1$ equal to the multiplicity
of $|xg_2|_n$ in $\Delta \cF$. Then 
$k_1-|yg_2|_n = (k_1-|xg_2|_n)+(k_2-x)$.
Rearranging gives $|xg_2|_n - |yg_2|_n = k_2 - x$. From Lemma~\ref{two-thirds-bound}, we have 
$$k_2-2t-1 \le |D(g_2,k_2,n) \cap \{1,\dots,k_2-1\}| < 2k_2/3,$$ or $k_2 \le 6t+2$.
Now, using \eqref{basic-equation},
$$\frac{k_2(k_2-1)n}{2tk_1} -  \frac{(t-1)n}{2k_1} = n >18k_2+36,$$
or, after some algebra,
\begin{equation}
\label{eq:n-upperbound}
k_2(k_2-1) > 36(k_2+2) \frac{tk_1}{n} +t(t-1).
\end{equation}
Using Proposition~\ref{beta-upper-bound} and that $t$ is a positive integer at least $\frac{1}{6} (k_2-2)$, \eqref{eq:n-upperbound} implies
$$k_2(k_2-1) > (k_2-2)(k_2+2),$$
or $k_2<4$.  Since we have assumed $k_2 \ge 4$, it must be the case that $n \le 18k_2+36$.
\end{proof}

\subsection{Lower bound on $n$ via maximum multiplicity}
\label{sec:n-lower}

Changing tactic, we next establish a lower bound on $n$ from the fact that no distance occurs too often.

\begin{prop}
\label{prop:n-lowerbound}
An Erd\H{o}s-deep pair of APs of lengths $k_1 \ge k_2 \ge 3$ satisfies $n \ge (k_2-t)^2/4$.
\end{prop}

\begin{proof}
Put $\mu= (k_2-t)/2$ and $m=\lfloor \mu \rfloor$. Suppose for contradiction that $n<\mu^2$.  We show that one of the $m$ most frequent shared distances, call it $d$, is at most $\mu$.  From our assumption on $n$, we may partition $\Z_n$ into $m$ intervals, each of size at most $m+1$.  By the pigeonhole principle, there exist two elements of $\ap_n(g_2,m+1)$ in the same interval, say $jg_2,j'g_2$, where $0 \le j<j' \le  m$.  Put $d:=|jg_2-j'g_2|_n$.  Since $k_2 \ge m+1$, the distance $d$ occurs in $\ap_n(g_2,k_2)$ with multiplicity at least $k_2-m$.  Moreover, we have $d \le m$ since $jg_2,j'g_2$ belong to the same interval; hence the multiplicity of distance $d$ in $\ap_n(1,k_1)$ is at least $k_1-m$.  The total multiplicity of $d$ in $\Delta \cF$ is at least $(k_1-m)+(k_2-m) \ge k_1+k_2-2\mu = k_1+t = k$, a contradiction to the Erd\H{o}s-deep property.
\end{proof}

At this point, we are able to reduce the classification to a finite problem.

\begin{thm}
\label{finitely-many}
There are at most finitely many Erd\H{o}s-deep pairs of APs with lengths $k_1,k_2>3$.
\end{thm}

\begin{proof}
Suppose $\{\ap_n(1,k_1),\ap_n(g_2,k_2)\}$ is an Erd\H{o}s-deep pair with $k_1 \ge k_2 >3$. 
Putting together Propositions~\ref{prop:n-upperbound} and \ref{prop:n-lowerbound} gives
$(k_2-t)^2/4 \le n  \le 18k_2+36$.  Using the upper bound on $t$ from Lemma~\ref{three-sevenths}, we have $\tfrac{4}{49}k_2^2 \le n \le 18k_2+36$.  This implies $k_2 \le 222$ and $n \le 4032$.  We also have $g_2, k_1 \le n/2$.
\end{proof}

We next note that our lower bound on $n$ above also implies a universal bound on $t=k-k_1$.
This simply comes from a quadratic inequality on our parameters.  For the following, put $\beta:=n/k_1$.

\begin{prop}
\label{t-upper-bound}
If $n \ge (k_2-t)^2/4$, then $t \le 5\beta-\tfrac{3}{4}$.
\end{prop}

\begin{proof}
Suppose for contradiction that $t > 5 \beta-\tfrac{3}{4}$.
Using \eqref{useful-param}, we have
$$n=\beta k_1 = \frac{\beta k_2(k_2-1)-\beta t(t-1)}{2t}.$$
So, from the assumed bound on $n$, we obtain the inequality
\begin{equation}
\label{k2-ineq}
2\beta k_2(k_2-1)-2\beta t(t-1)- t(k_2-t)^2 \ge 0.
\end{equation}
The right side is a quadratic in $k_2$ with leading coefficient $2\beta - t < 0$
and zeros at $k_2=t,u$, where $u=\frac{t^2+2\beta(t-1)}{t-2\beta}>t$.

By Lemma~\ref{three-sevenths}, we have $t \le \frac{3}{7}(k_2-1)$.  Also, \eqref{k2-ineq} implies $k_2 \le u$, so
$t \le \frac{3}{7}(u-1)$.  Simplifying this, we get
$t-2\beta \le \frac{3}{7}(t+2\beta -1)$, or equivalently
$t \le 5\beta-\tfrac{3}{4}$.
\end{proof}

\begin{cor}
If $\{\ap_n(g_1,k_1),\ap_n(g_2,k_2)\}$ is an Erd\H{o}s-deep pair with $\binom{k}{2}=\binom{k_1}{2}+\binom{k_2}{2}$, then $t:=k-k_1 \le 29$.
\end{cor}

\begin{proof}
If $k_1=k_2=3$, then $k=4$ and $t=1$.  Otherwise,  Propositions~ \ref{beta-upper-bound}, \ref{prop:n-lowerbound} and \ref{t-upper-bound} apply.  Substitute $\beta<6$ to obtain $t \le 5\beta-\frac{3}{4} < 29.25$.  Since $t$ is an integer, $t \le 29$.
\end{proof}

\subsection{Computer search}

Theorem~\ref{finitely-many} reduces our classification problem to a finite one, and the results above give a further reduction on the size of the parameter space to check.  We wrote a simple computer program in both Java and Python to confirm that Examples~\ref{edfam-13}, \ref{edfam-19} and \ref{edfam-31} give the only parameter tuples $(n,k_1,k_2)$ with $k_1,k_2>3$ producing Erd\H{o}s-deep pairs of APs.

The strategy of the computer search is to loop over $t \in \{1,2,\dots,29\}$ and then $k_2$ satisfying 
\begin{itemize}
\item
$(k_2-t)^2/4 \le 18k_2+36$, from Propositions~\ref{prop:n-upperbound} and \ref{prop:n-lowerbound}; and
\item
$t<\frac{3}{7}(k_2-1)$, from Lemma~\ref{three-sevenths}.
\end{itemize}
For each pair $(t,k_2)$, use \eqref{useful-param} to compute $k_1$ and $k$.  In the case that $k_1,k \in \Z$, we loop over $n$ in the range specified by its bounds in Sections~\ref{sec:n-upper}-\ref{sec:n-lower}, and then over $g_2 \in \{1,2,\dots,\lfloor n/2 \rfloor\}$.  
We remind the reader that $g_1$ can be assumed to equal $1$ after multiplication by $g_1^{-1} \pmod{n}$.  The loop on $g_2$ can be further reduced to
$k_2-t \le g_2 < 2k_1$; this is not crucial for the feasibility of the search, but the  second author's thesis \cite{thesis} can be consulted for further details.
%\begin{prop}
%If $\{\ap_n(1,k_1),\ap_n(g_2,k_2)\}$ is an Erd\H{o}s-deep pair, then $k_2-t \le g_2 < 2k_1$.
%\end{prop}

Finally, we remark on a programming trick to speed up the search.  Since we expect most parameter lists $t,k_2,n,g$ to (spectacularly) fail to form the desired multiplicities, we break out of the multiplicity check as soon as either (a) more than $k-1$ distinct distances are found, or (b) some multiplicity exceeds $k-1$. With this and the additional bounds on $n$ and $g_2$, searching the needed region can be performed in less than 30 minutes on an average desktop computer.

\section{A construction for families of square size}
\label{construction}

In this section, we give a construction to show that, when $s$ is a square integer, there are infinitely many tuples $(k_1,\dots,k_s)$ of set sizes that admit an Erd\H{o}s-deep family of $s$ sets.  This sharply contrasts the situation we have studied for $s=2$, where only four tuples $(k_1,k_2)$ with $k_1 \ge k_2$ can yield an Erd\H{o}s-deep family.

Let $h$ and $\ell$ be positive integers.  We note an identity on triangular numbers:
\begin{equation}
\label{square-s-identity}
\binom{h \ell}{2} = \binom{h+1}{2} \binom{\ell}{2} + \binom{h}{2} \binom{\ell+1}{2}.
\end{equation}

The following construction 
produces an Erd\H{o}s-deep family of $h^2$ modular arithmetic progressions whose multiplicities combine according to \eqref{square-s-identity}.

\begin{thm}
\label{thm:square-s}
Let $h,\ell$ be integers with $\ell \ge 3$ and $h \ge 2$.  Suppose $D=\{g_1,\dots,g_h\} \subseteq \Z_n$ has the property that each $|jg_i|_n$ is distinct for $i\in [h]$ and $j \in [\ell]$.
For each $i$, let $\cF_i$ be the family consisting of $i$ copies of $\ap_n(g_i,\ell)$ together with $h-i$ copies of $\ap_n(g_i,\ell+1)$.
Then $\cF=\cup_{i=1}^h \cF_i$ is an Erd\H{o}s-deep family of $s=h^2$ sets and $\binom{k}{2}$ distances, where $k=h\ell$.
\end{thm}

\begin{proof}
Looking at each $\cF_i$ separately, the difference $g_i$ occurs $i(\ell-1)+(h-i)\ell = h \ell -i$ times within some AP.  Similarly, $|jg_i|_n$, $j \in \{1,\dots,\ell\}$, occurs as a distance in $\cF_i$
with total multiplicity $i(\ell-j)+(h-i)(\ell+1-j) = h(\ell+1-j)-i$. Note that this is zero when $(i,j)=(h,\ell)$, since $\cF_h$ simply consists of $h$ copies of $\ap_n(g_h,\ell)$ (and no APs of length $\ell+1$).

From our assumption on $D$, the set of multiplicities occurring in $\Delta \cF$ is
$$\{h(\ell+1-j)-i : (i,j) \in [h] \times [\ell] \setminus \{(h,\ell)\} \} = [k-1].$$
It follows that $\cF$ is an Erd\H{o}s-deep family with $\binom{k}{2}$ distances.
\end{proof}

We remark that the required condition on $D$ is in practice very easy to fulfill when $n$ is large relative to $h$ and $\ell$.  For a fixed $h$, each choice of $\ell$ gives a geometric infinite Erd\H{o}s-deep family of size $s=h^2$ in $\Z_n$ for sufficiently large $n$.  This is in sharp contrast with Theorem~\ref{main}, in which it is shown that (up to scaling and translation) only one infinite family exists for $s=2$.

We illustrate the construction with an example.

\begin{ex}
Let $n \ge 13$, and take $h=2$, and $\ell=3$.  Choose $D=\{g_1,g_2\}$ where $g_1=1$ and $g_2=4$.  The construction of Theorem~\ref{thm:square-s} builds the family
$$\cF = \{\{0,1,2,3\},\{0,1,2\},\{0,4,8\},\{0,4,8\}\},$$
which has distance multiset $\{1^5,2^3,3^1,4^4,8^2\}$.  Since $|8|_{13}=5$, this set has the additional feature that the distinct {\bf distances} (as well as the multiplicities) form the interval $\{1,2,\dots,5\}$ when $n=13$.  Such families are called \emph{Winograd-deep} and discussed further in \cite{DistGeom,thesis}.
\end{ex}

\section{Discussion}
\label{sec:discussion}
Our main result is that Erd\H{o}s-deep pairs of APs in $\Z_n$ exist if and only if $k_1=k_2=3$ for all $n \ge 7$, and in three sporadic cases $(k_1,k_2)=(6,4)$, $(7,6)$, and $(11,9)$, each of which exist for a single value of $n$.  This is the case of $s=2$ APs in an Erd\H{o}s-deep family.  By contrast, Theorem~\ref{thm:square-s} says that square values of $s$ admit constructions for infinitely many tuples $(k_1,\dots,k_s)$, and in particular this holds for $s=4$.  

The classification of Erd\H{o}s-deep triples of APs, (that is, the case $s=3$) is likely difficult to fully settle.  After a computer search of small parameters, we have a conjectured classification.

\begin{conj}
\label{s3}
An Erd\H{o}s-deep family of three APs of lengths $k_1 \ge k_2 \ge k_3$ in $\Z_n$ exists if and only if $(k_1,k_2,k_3) \in \{(4, 4, 3), (6, 3, 3) \}$, each for infinitely many $n$, or
\begin{align*}
(k_1,k_2,k_3) \in  \{& (6, 5, 3),  (6, 6, 4), (6, 6, 6), (7, 7, 3), (9, 4, 3), (8, 7, 4), (8, 8, 5), \\
&(10, 6, 4), (12, 4, 4), (13, 5, 3), (13, 7, 4), (16, 5, 4), (21, 6, 4) \},
\end{align*} 
each for a finite number of values of $n$. 
\end{conj}

The geometric infinite families in Conjecture~\ref{s3} are realized for $n \ge 15$ by 
$$\cF=\{\{0,1,2,3\},\{0,3,6,9\},\{0,1,2\}\},~\text{with}~\Delta \cF = \{1^5, 2^3, 3^4, 6^2, 9^1\},$$
and
$$\cF=\{\{0,1,2,3,4,5\},\{0,4,8\},\{0,4,8\}\},~\text{with}~\Delta \cF = \{1^5, 2^4, 3^3, 4^6,5^1, 8^2\}.$$

A potentially interesting relaxation is to consider families of $s \ge 3$ modular APs which achieve an interval of distance multiplicities of the form $\{a,a+1,\dots,b\}$ for $a \ge 2$.  We have not undertaken a systematic study, but we offer two examples of this occurrence.

\begin{ex}
Let $n=29$ and consider $\cF=\{\ap_n(1,14),\ap_n(4,5), \ap_n(12,3)\}$.  The first AP has distance $i$ occurring with multiplicity $14-i$ for $i=1,\dots,13$.  Distances from the other two APs combine to produce an interval of multiplicities from $2$ to $14$. Specifically, we have
$$\Delta\cF = \{13^2,11^3,10^4,9^5,12^6,7^7,6^8,8^9,5^{10},3^{11},2^{12},1^{13},4^{14}\}.$$
\end{ex}

\begin{ex}
Let $n=17$ and $\cF=\{\ap_n(1,7),\ap_n(3,6),\ap_n(2,5),\ap_n(4,3)\}$.  Distance multiplicities are shown in Table~\ref{tab:interval-example}
\begin{table}[htbp]
\begin{center}
\begin{tabular}{l|cccccccc}
distance & 1&2&3&4&5&6&7&8\\
\hline
mult in $\ap_{17}(1,7)$ & 6&5&4&3&2&1&&\\
mult in $\ap_{17}(3,6)$ & &1&5& &2&4& &3\\
mult in $\ap_{17}(2,5)$ & &4& &3& &2& &1\\
mult in $\ap_{17}(4,3)$ & & & & 2 & & & & 1\\
\hline
total in $\Delta \cF$ & 6 & 10 & 9 & 8 & 4 & 7 & &  5\\
\end{tabular}
\medskip
\caption{A family of four APs with an interval of nonzero distance multiplicities}
\label{tab:interval-example}
\end{center}
\end{table}
\end{ex}
Given any family $\cF$ of APs (or more general subsets) in $\Z_n$, it is possible to explore statistics on the multiplicities in $\Delta \cF$.  Erd\H{o}s-deep families are extremal in that they maximize the number of distinct multiplicities, assuming $\sum_{A \in \cF} \binom{|A|}{2}$ is a triangular number. Families can exist with a smaller number of still distinct distance multiplicities. Single sets with this property are considered in \cite{DistGeom}.  In general, if $n$ and the set sizes in $\cF$ are prescribed, it may be interesting to quantify how close the distance multiplicities can get to satisfying the Erd\H{o}s-deep condition.

We conclude by expanding briefly on the connection to music mentioned in the introduction.  Each set $A \subseteq \Z_n$ can be associated with a rhythm spanning $n$ units of time, \cite{Toussaint}.  The elements of $A$ correspond to the time stamps of note-hits, also known as `onsets'.  With this correspondence in mind, it is argued in \cite{DistGeom} that Erd\H{o}s-deep sets, and modular APs in general, lead to interesting rhythms.  Examples of rhythms from a variety of cultures are shown to have the properties of an Erd\H{o}s-deep set.  One of the motivations for our generalization to families of modular APs was the possibility that they could model multi-voiced rhythms.  A further discussion of this application, including a selection of compositions based on this structure, can be found in Chapter 6 of the second author's thesis, \cite{thesis}.

%families of geometric sets of points in the plane; possible to use the square-s construction here?

\section*{Acknowledgment}

This research is supported by NSERC Discovery Grant \#312595--2017.

\end{document}